\documentclass[12pt,reqno]{amsart}
\usepackage{amsmath,amssymb,amsthm}
\usepackage{color}
\usepackage[unicode]{hyperref}
\usepackage{longtable}
\ifx\pdfoutput\undefined\else
\hypersetup{
colorlinks=true,
linkcolor=blue,
citecolor=blue,
urlcolor=blue,
filecolor=blue,
bookmarksnumbered=true,
pdfstartview=FitH,
pdfhighlight=/N
}
\fi
\newtheorem{theorem}{Theorem}

\newtheorem{proposition}{Proposition}
\theoremstyle{remark}
\newtheorem{remark}{Remark}
\newtheorem{example}{Example}
\newtheorem*{acknowledgements}{Acknowledgements}
\renewcommand{\d}{{\mathrm d}}

\newcommand{\qede}{\hspace*{\fill}$\Diamond$\medskip}

\hypersetup{pdfauthor={Wan, Zucker},pdftitle={}, bookmarks=true, pdfstartview={FitV}}

\def \i1{\int_0^1}
\def \f#1#2{{#1\over{#2}}}
\def \s#1{\sqrt{#1}}
\def \fr{\textstyle\f}
\def \lm{\lambda}
\def \k{\left(\f{K'}{K}\right)^{s-1}\d k}

\begin{document}

\title{Integrals of $K$ and $E$ from Lattice Sums} 

\author{J. G. Wan}
\address{Engineering Systems and Design, Singapore University of Technology and Design, 20 Dover Drive Singapore 138682}
\email{james\_wan@sutd.edu.sg}

\author{I. J. Zucker}
\address{Department of Physics, Kings College London, Strand, London WC2R 2LS, UK}
\email{jzucker@btinternet.com}

\date{\today}


\begin{abstract}
We give closed form evaluations for many families of integrals, whose integrands contain algebraic functions of the complete elliptic integrals $K$ and $E$. Our methods exploit the rich structures connecting complete elliptic integrals, Jacobi theta functions,  lattice sums, and Eisenstein series. Various examples are given, and along the way new (including 10-dimensional) lattice sum evaluations are produced.

\end{abstract}

\maketitle

\section{Introduction}
\label{sec-intro}

The complete elliptic integrals of the first and second kind, $K$ and $E$ respectively, are defined by
\begin{align} \nonumber 
K(k) & =\i1\f{\d x}{\s{(1-x^2)(1-k^2x^2)}}  = \frac{\pi}{2}\, {_2}F_1\left(\frac{1}{2},\frac{1}{2};1;k^2\right), \\
 E(k) & =\i1\s{\f{1-k^2x^2}{1-x^2}}\d x =\frac{\pi}{2}\, {_2}F_1\left(-\frac{1}{2},\frac{1}{2};1;k^2\right). \label{eq1.1}
\end{align}
As usual, $_pF_q$ denotes the generalized hypergeometric function \cite[Ch.~2]{aar}. In equation \eqref{eq1.1}, $k$ is known as the modulus;  the complimentary functions $K'$ and $E'$ are the same functions with argument $k' = \sqrt{1-k^2}$. The derivatives of $K$ and $E$ can be expressed as follows:
\begin{equation} \label{derivke}
\frac{\mathrm d K(k)}{\mathrm d k} = \frac{E(k)-k'^2 K(k)}{k k'^2}, \quad \frac{\mathrm d E(k)}{\mathrm d k} = \frac{E(k)-K(k)}{k}.
\end{equation}

In recent years, integrals of the form
\begin{equation} \label{eq1.2}
\i1 f(k)\, K^mK'^nE^pE'^r \d k
\end{equation}
have appeared in the context of multi-loop Feynman diagrams, Ising-type
integrals \cite{bbbg}, random walks \cite{walk2}, Mahler measures, and non-critical $L$-values of modular forms \cite{prs, ramazeta}.  More intricate integrals containing $K$ also appear in connection with lattice Green's functions \cite{Lbook, watson}. It is thus our aim to evaluate  the integrals \eqref{eq1.2} whenever possible.

This is in general challenging. For instance, it was experimentally discovered in \cite{wan} that
\begin{equation}\label{first}
 \int_0^1 K'(k)^3 \d k =  5 \int_0^1 k K'(k)^3 \d k = \frac{\Gamma^8\left(\frac14\right)}{128\pi^2}.
\end{equation}
Despite its simplicity in form, \eqref{first} was only proved in \cite{rwz} and \cite{zhou}; for the latter proof, Zhou developed, among other things, a generalization of the classical Clebsch-Gordan coefficients.

In this paper, we show how infinite families of such integrals may be found in closed forms when certain multiple (lattice) sums are expressed as Mellin transforms of products of Jacobi $\theta$ functions.  To this end, we draw upon the deep connections between $\theta$ functions and complete elliptic integrals \cite[Ch.~2]{AGM}.
The Jacobi $\theta$ functions are
\begin{align} \nonumber
\theta_2(q)&=\sum_{n=-\infty}^{\infty} q^{(n-1/2)^2}=\s{\f{2kK(k)}{\pi}}, \\ \nonumber
\theta_3(q)&=\sum_{n=-\infty}^{\infty} q^{n^2}=\s{\f{2K(k)}{\pi}}, \\
\theta_4(q)&=\sum_{n=-\infty}^{\infty}(-1)^n q^{n^2}=\s{\f{2k'K(k)}{\pi}}, \label{eq1.3}
\end{align}
\begin{equation} \label{qderiv}
\text{where} \qquad  q=e^{-\pi K'(k)/K(k)}, \quad \text{and} \quad \frac{\d q}{\d k} = \frac{\pi^2 q}{2kk'^2 K(k)^2}. 
\end{equation}

Said multiple sums can be expressed as products of Dirichlet $L$-series with real characters \cite[Ch.~4]{Lbook}. Explicitly, the series that occur here are
\begin{align} \nonumber
\zeta(s) & =1+2^{-s}+3^{-s}+4^{-s}+\cdots, \quad \beta(s)=1-3^{-s}+5^{-s}-7^{-s}+\cdots, \\ \nonumber
\lm(s) &= (1-2^{-s})\zeta(s), \qquad \qquad \qquad \ \ \eta(s) =(1-2^{1-s})\zeta(s), \\ \nonumber
L_{-3}(s) & = 1 -2^{-s}+4^{-s}-5^{-s}+\cdots, \\  \nonumber
L_{-8}(s) & =1+3^{-s}-5^{-s}-7^{-s}+\cdots, \quad L_{8}(s) =1-3^{-s}-5^{-s}+7^{-s}+\cdots, \\ \nonumber
L_{12}(s) & =1-5^{-s}-7^{-s}+11^{-s}+\cdots, \\ \nonumber
L_{-24}(s) & =  1+5^{-s}+7^{-s}+11^{-s}-13^{-3}-17^{-s}-19^{-s}-23^{-s}+\cdots, \\
L_{24}(s) & =  1+5^{-s}-7^{-s}-11^{-s}-13^{-3}-17^{-s}+19^{-s}+23^{-s}+\cdots . \label{eq1.4}
\end{align}
The first four are respective the Riemann zeta,  Dirichlet beta, Dirichlet lambda, and Dirichlet eta functions; $\beta(2)$ is Catalan's constant and is denoted by $G$.

In the next section, we utilize the connecting formulae \eqref{eq1.3} and \eqref{qderiv} and known lattice sum evaluations to produce new $K$ integrals. In sections \ref{sec:eisen} and \ref{sec:mod}, we evaluate more $K$ integrals using Eisenstein series and some theory of modular forms; these original results are summarized in Theorems \ref{thmkeven} and \ref{thmkodd}. In the final section, we also look at integrals containing $E$.

\section{Transformation into lattice sums}

We begin with an example  of converting a multiple sum into an equation of the form \eqref{eq1.2}.
\begin{example} \label{eg1}
Consider the double sum
\begin{equation} \label{eq1.5}
\sum_{m_1,m_2}\f{(-1)^{m_2}}{\left[(m_1-1/2)^2+m_2^2\right]^s}, 
\end{equation}
where the summation is performed over all values of the indices from $-\infty$ to $\infty$. The normalized \textit{Mellin} transform operator is
\begin{equation}  \label{mellin}
M_s[f(t)]=\f{1}{\Gamma(s)}\int_0^{\infty}f(t)t^{s-1}\d t, 
\end{equation}
in particular, $M_s[e^{-Nt}] = N^{-s}$. Accordingly, equation \eqref{eq1.5} may be written as
\begin{align} \nonumber
 \sum_{m_1,m_2} \f{(-1)^{m_2}}{\left[(m_1-1/2)^2+m_2^2\right]^s} &= M_s \Big[ \sum_{m_1,m_2} e^{-(m_1-1/2)^2t}(-1)^{m_2}e^{-m_2^2t} \Big] \\
&=\f{1}{\Gamma(s)}\int_0^\infty \theta_2(e^{-t})\theta_4(e^{-t})t^{s-1}\d t. \label{eq1.6}
\end{align}
In \eqref{eq1.6}, let $t=\pi K'(k)/K(k)$, and substitute in the expressions in \eqref{eq1.3} and \eqref{qderiv}. We obtain
\begin{equation} \label{eq1.7}
\sum_{m_1,m_2}\f{(-1)^{m_2}}{\left[(m_1-1/2)^2+m_2^2\right]^s}=\f{\pi^s}{\Gamma(s)}\i1 
 \f{1}{k^{1/2} k'^{3/2}K(k)}\left(\f{K'(k)}{K(k)}\right)^{s-1}\d k.
\end{equation}
Now it has been shown in \cite{zucker} that
\begin{equation} \label{eq1.8}
\sum_{m_1, m_2} \f{(-1)^{m_2}}{\left[(m_1-1/2)^2+m_2^2\right]^s}=2^{2s+1}L_{-8}(s)L_8(s),
\end{equation}
which is equivalent to the Mellin transform of the identity 
\[ \theta_2(q)\theta_4(q) = 2q^{1/4}\sum_{n=0}^\infty (-1)^n \bigg[\frac{q^n}{1+q^{4n+1}}-\frac{q^{3n+2}}{1+q^{4n+3}}\bigg].\]
Thus, \eqref{eq1.7} simplifies to
\begin{equation} \label{eq1.9}
\i1 \f{1}{k^{1/2} k'^{3/2}K(k)}\left(\f{K'}{K}\right)^{s-1}\d k=\pi^{-s}\Gamma(s)2^{2s+1}L_{-8}(s)L_8(s). 
\end{equation}
For instance, when $s=1$, $L_{-8}(1)=\pi/\s{8}$ and $L_8(1)=\log(1+\s{2})/\s{2}$; so we have
\begin{equation} \label{eq1.10}
\i1 \f{\d k}{k^{1/2} k'^{3/2}K(k)}=2\log(1+\s{2}).
\end{equation}
As results such as \eqref{eq1.8} are readily available in \cite{zucker}, many expressions such as \eqref{eq1.9} may be found. \qede
\end{example}

A more general analysis is now given. Suppose a multiple sum $L(m,n,p;s)$ can be written in terms of a Mellin transform of products of $\theta$ functions thus:
\begin{equation} \label{eq1.11}
L(m,n,p;s)=\f{1}{\Gamma(s)}\int_0^\infty \theta_2^m(e^{-t})\theta_3^n(e^{-t})\theta_4^p(e^{-t})t^{s-1}\d t.
\end{equation}
As we have seen in Example \ref{eg1}, $\theta_2$ corresponds to a component in the multiple sum of the form $1/(N-1/2)^2$, $\theta_3$ corresponds to $1/N^2$, and $\theta_4$ corresponds to $(-1)^N/N^2$. Transforming the integral for $L(m,n,p;s)$ using \eqref{eq1.3} and \eqref{qderiv} as in Example \ref{eg1}, we obtain the following general result.
\begin{proposition} \label{prop1}
When the integral on the right converges, the following equation holds, and provides an analytic continuation for the sum $L(m,n,p;s)$. (The argument of $K$ or $K'$ is omitted when it is clear from the context.)
\begin{equation} \label{eq1.12}
L(m,n,p;s)=\f{\pi^s}{\Gamma(s)}\left(\f{2}{\pi}\right)^{\fr{m+n+p-2}{2}}\i1 k^{\fr{m-2}{2}}k'^{\fr{p-4}{2}}
K^{\fr{m+n+p-2s-2}{2}}K'^{s-1}\d k.
\end{equation}
As a consequence of the fact that the Poisson transform of $\theta_2^m(e^{-t})\theta_3^n(e^{-t})\theta_4^p(e^{-t})$ is $\theta_2^p(e^{-t})\theta_3^n(e^{-t})\theta_4^m(e^{-t})$, we also have
\begin{equation} \label{eq1.13}
L(m,n,p;s)=\f{\pi^s}{\Gamma(s)}\left(\f{2}{\pi}\right)^{\fr{m+n+p-2}{2}}\i1 k^{\fr{p-2}{2}}k'^{\fr{m-4}{2}}
K'^{\fr{m+n+p-2s-2}{2}}K^{s-1}\d k.
\end{equation}
\end{proposition} \smallskip

Equating the two results in Proposition \ref{prop1} gives
\begin{equation} \label{prop1eq} \pi^{(m+n+p-4s)/2}L(m,n,p;s) = L(p,n,m;(m + n + p- 2s)/2).\end{equation}
Moreover, writing $k'^N$ as $k'^{N-2}-k^2 k'^{N-2}$ and applying \eqref{eq1.12} to each term, we get the beautiful equation
\[ L(m,n,p+4;s)-L(m,n+4,c;s) + L(m+4,n,p;s) = 0, \]
which is Jacobi's celebrated result  $\theta_3^4=\theta_2^4+\theta_4^4$ in disguise.

In Table \ref{table1}, results obtained using Proposition \ref{prop1} and sums of up to dimension 8 are given; the lattice sum evaluations come from  $\theta$ function identities recorded in \cite{AGM, zucker}, and \cite[Ch.~1 and 6]{Lbook}.   A selection of some interesting formulae which emerge are shown below.  

\begin{example}
Taking $s=m=p=2$, $n=0$ in Proposition \ref{prop1}, we obtain
\[ \int_0^1 \frac{K'}{K} \frac{\d k}{k'}= 2G. \] 
The lattice sum involved is $\sum_{m_i} (-1)^{m_3+m_4}[(m_1-1/2)^2+(m_2-1/2)^2+m_3^2+m_4^2]^{-2}$ and the final evaluation comes from \cite{zucker}. 

Taking $m=2, \, n=0, \, p=4$ and $s=3/2$ in Proposition \ref{prop1}, we obtain
\[ \int_0^1 \sqrt{KK'}\mathrm dk = \sqrt2 \, \beta\Bigl(\frac32\Bigr)\lambda\Bigl(\frac32\Bigr), \]
with \cite{zucker} again giving the final evaluation. 
\qede
\end{example}

\begin{example}
From the 8-fold sum \cite{zucker} $L(4,0,4;s) = $ 
\begin{align*}
& \sum_{m_i}\f{(-1)^{(m_5+m_6+m_7+m_8)}}{\left[(m_1-\fr{1}{2})^2+(m_2-\fr{1}{2})^2+(m_3-\fr{1}{2})^2+(m_4-\fr{1}{2})^2+m_5^2+m_6^2+m_7^2+m_8^2\right]^s} \\
= & \ 16\eta(s-3)\lambda(s), 
\end{align*}
we derive
\begin{equation} \label{t2t4} \i1 kK^2\left(\f{K'}{K}\right)^{s-1} \d k=2\pi^{3-s}\Gamma(s)\eta(s-3)\lambda(s), \end{equation}
which yields
\[ \i1 kK^2 \d k=\f{7}{4}\zeta(3), \ \ {\int_0^1} kKK' \d k= \f{\pi^3}{16},\ \ \i1 kK'^2 \d k=\f{7}{4}\zeta(3),\ \
\i1 k \f{K'^3}{K} \d k=\f{\pi^3\log 2}{8}. \]
It is interesting that \cite{wan} gives the first three results  above by entirely different methods. \qede
\end{example}

\begin{remark}
In Proposition \ref{prop1}, if $m=0$, then for convergence reasons we could replace the integrand in \eqref{eq1.11} by $(\theta_3^n(e^{-t})\theta_4^p(e^{-t})-1)t^{s-1}$. The same steps could be used to convert this into a $K$ integral, though the final result is not as elegant. As an example, when $n=s=2$ and $m=p=0$, 
\[ \int_0^1 \frac{K (2K'-\pi)}{k  k'^2 K'^3}\mathrm dk = \frac{2}{\pi^2} \sum_{(m_1,m_2) \ne 0} \frac{1}{(m_1^2+m_2^2)^2}=\frac{4G}{3},\]
the second equality being a special case of the Hardy-Lorenz sum.

Instead of subtracting by 1 as above, we could alternatively subtract two sums, in both of which $m=0$. An example of this can be found in Table \ref{table1}, see the entry $\theta_4^6-\theta_3^2\theta_4^4$. 
\qede
\end{remark}

The next example uses known $K$ integrals to provide a new lattice sum evaluations; this idea is further explored subsequently.
\begin{example}
 Proposition \ref{prop1} equates the value of the multiple sum $L(2,0,6;3)$ to the integral of $k' K'^2$, which can be found using the Fourier series method in \cite[Sec.~6]{wan}. We thereby obtain the 8-dimensional sum evaluation
\[  \sum_{m_i} \frac{(-1)^{m_3+m_4+\cdots+m_8}}{\big[(m_1-\frac12)^2+(m_2-\frac12)^2+m_3^2+m_4^2+\cdots+m_8^2\big]^3} = \frac{\pi^3}{4}\biggl[1+2\, {_4F_3}\biggl({{\frac12,\frac12,\frac12,\frac12}\atop{1,1,1}};1\biggr)\biggr]. \] 
As a more involved example, it is possible to show that 
\[ \frac14 L\Big(\frac52,0,\frac72;2\Big)=\int_0^1 \Big(\frac{k}{k'}\Big)^{3/4}K\mathrm dk = \frac{\pi^2}{12}\sqrt{5+\frac{1}{\sqrt2}}. \]
This leads to the lattice sum
\[ \sum_{m_i} \frac{(-1)^{m_3+m_4+m_5+m_6}}{\big[(m_1-\frac12)^2+(m_2-\frac12)^2+2(m_3-\frac14)^2+m_4^2+m_5^2+m_6^2\big]^2} = \frac{\pi^2}{6}\sqrt{10+\sqrt2}, \]
where we have applied the $\theta$ function identities \cite[equation (1.3.8)]{Lbook}.
\qede
\end{example}

\subsection{Relationships between lattice sums}

Proposition \ref{prop1} leads to various non-trivial relations among lattice sums, the next result being one useful example.
\begin{proposition} \label{prop2} The following relationships hold, when $L$  is viewed as an analytic continuation  (if necessary) provided by Proposition \ref{prop1}:
\begin{align} \label{prop2eq1}
L(2m,n,n;s) & = 2^{m-s}\, L(m,m,2n;s) \\
 & = \pi^{2s-m-n} L(n,n,2m; m+n-s).  \label{prop2eq2}
\end{align}
\end{proposition}
\begin{proof}
Starting with \eqref{eq1.13},
\[ L(2m,n,n;s) = \f{\pi^s}{\Gamma(s)}\left(\f{2}{\pi}\right)^{m+n-1}\i1 k^{\fr{n-2}{2}}k'^{m-2}
K'^{m+n-1-s}K^{s-1}\d k, \]
we make the change of variable $k \mapsto (1-k)/(1+k)$. The new integrand can be simplified using the quadratic transformations \cite[Ch.~1]{AGM}
\begin{equation} \label{k1}
K'(k) = \frac{2}{1+k} K\bigg(\frac{1-k}{1+k}\bigg), \quad K(k) = \frac{1}{1+k}K\bigg(\frac{2\sqrt k}{1+k}\bigg).
\end{equation}
The simplified integral is identified as $2^{m-s}\, L(m,m,2n;s)$,  using \eqref{eq1.12}, thus the equality \eqref{prop2eq1} is obtained. Equation \eqref{prop2eq2} then follows using \eqref{prop1eq}.
\end{proof}

\begin{example}[New lattice sum evaluations] \label{eg5}
Proposition \ref{prop2} shows that 
\[ 2\, L(2,2,2;2) = L(1,1,4;2). \] The value of the left hand side is given in \cite[Example 2]{rwz}, and thus, writing out the right hand side as a convergent 6-dimensional lattice sum, we obtain the apparently new evaluation
\begin{equation} \label{en6}
 \sum_{m_i} \frac{(-1)^{m_3+m_4+m_5+m_6}}{\big[(m_1-\frac12)^2+m_2^2+m_3^2+m_4^2+m_5^2+m_6^2\big]^2} = \frac{\Gamma^4(\frac14)}{2\pi}. 
\end{equation}
A special case of Jacobi's triple product identity states that
\begin{equation} \label{triple}
 \theta_2 \theta_3 \theta_4 = \theta_1' = \sum_n 2n(-1)^n  q^{(n+1/2)^2},
\end{equation}
so the above evaluation of $L(1,1,4;2)$ may be rewritten as
\[ \sum_{m_i} \frac{m_1\, (-1)^{m_1+m_2+m_3+m_4}}{\big[(m_1+\frac12)^2+m_2^2+m_3^2+m_4^2\big]^2} = \frac{\Gamma^4(\frac14)}{4\pi}. \] 

Similarly, Proposition \ref{prop2} gives $2L(4,2,2;3) = L(2,2,4;3)$. The evaluation of the left hand side is given in \cite[Proposition 1]{wan}, and therefore we obtain an evaluation for the right hand side sum:
\begin{align}\nonumber
& \sum_{m_i} \frac{(-1)^{m_5+m_6+m_7+m_8}}{\big[(m_1-\frac12)^2+(m_2-\frac12)^2+m_3^2+m_4^2+m_5^2+m_6^2+m_7^2+m_8^2\big]^3} \\
= & \ \frac{\pi^4}{4}\,{_7F_6}\biggl({{\frac54,\frac12,\frac12,\frac12,\frac12,\frac12,\frac12}\atop{\frac14,1,1,1,1,1}};1\biggr).
\end{align}
\qede
\end{example}

\begin{remark}
We can evaluate $L(1,1,1;s)$ by applying the Mellin transform \eqref{mellin} to both sides of \eqref{triple}. Similarly, $L(1/3,1/3,1/3;s)$ can be found using the following version of Euler's pentagonal number theorem,
\[ \Big[\frac12 \theta_2 \theta_3 \theta_4\Big]^{1/3}= q^{1/12}\prod_{n=1}^\infty (1-q^{2n}) = \sum_n (-1)^n q^{(6n-1)^2/12}. \]
Both evaluations are recorded in Table \ref{table1}.

For $L(2,2,2;s)$, we may use the following series due to Hirschhorn \cite{hirschhorn},
\[ 2(\theta_2 \theta_3 \theta_4)^2 = \sum_{m,n} (-1)^{m+n} \, \mathrm{Re}\big[(10m+3+i(10n+1))^2\big] q^{((10m+3)^2+(10n+1)^2)/20}, \]
resulting in the evaluation 
\[ \mathrm{Re}\, \sum_{m,n} \frac{(-1)^{m+n}}{\big[10m+3 + i(10n+1)\big]^2} = \frac{1}{50}\int_0^1 \frac{K}{k'}\d k= \frac{\Gamma(\frac14)^4}{800\pi}.\]
\qede
\end{remark}

\begin{remark}
Proposition \ref{prop2} also leads to a number of functional equations for specific sums, for instance
\begin{align*}
(2\pi)^{2m-2s} L(m,m,2m;s) & = L(m,m,2m;2m-s), \\
(2\pi)^{3m-2s} L(m,m,4m;s) & = L(m,m,4m;3m-s).
\end{align*}
Following a similar proof procedure for Proposition \ref{prop2}, we may obtain many linear relations among sums, one example being
\[ 2^{s-m} L(2m,n,n+2;s)  = L(m,m+2,2n;s) - L(m+2,m,2n;s). \]
\qede
\end{remark}

\begin{example}
In \cite[proof of Theorem 1.1]{prs}, it is shown that for the function $f(\tau) = \eta^4(2\tau)\eta^4(4\tau)$ (here $\eta$ stands for the Dedekind eta function), an $L$-value yields the closed form
\begin{equation} \label{Lf4}
 L(f,4) = \frac{\pi^4}{192} \bigg[ {_5F_4}\bigg({{\frac12,\frac12,\frac12,\frac12,\frac12}\atop{1,1,1,\frac32}};1\bigg) + \frac{7\zeta(3)}{\pi^2}\bigg].
\end{equation}
In our notation, $L(f,4) = \frac{1}{16} L(4,2,2;4)$. Proposition \ref{prop2} gives $L(4,2,2;4) = \frac14 L(2,2,4;4)$, so we produce the following new 8-dimensional sum evaluation,
\[ \sum_{m_i} \frac{(-1)^{m_5+m_6+m_7+m_8}}{\big[(m_1-1/2)^2+(m_2-1/2)^2+m_3^2+m_4^2+\cdots+m_8^2\big]^4} = 64 L(f,4). \]
The corresponding $K$ integrals are
\[ \int_0^1 \frac{K'^3}{K}\mathrm d k = 4 \int_0^1 \frac{K^3}{K'}\mathrm d k = \frac{48}{\pi} L(f,4), \]
where $L(f,4)$ is given by \eqref{Lf4}. \qede
\end{example}

\subsection{10-dimensional sums} \label{sec:10d} Prior to this work, it seems that closed-form evaluations of lattice sums have been limited to dimensions 8 or less. In this section we give some 10-dimensional evaluations, namely equations \eqref{10d1} and \eqref{10d3}. Later, Remark \ref{rem5} provides another example.

One of the key results in \cite{rwz} is the evaluation \eqref{first}.
When translated into multiple sums using Proposition \ref{prop1}, it gives
\[ \frac{\Gamma^8\left(\frac14\right)}{48\pi^2} = 5 L(4,2,4;4) = L(2,4,4;4) = \frac18 L(1,1,8;4), \]
where we have used Proposition \ref{prop2} for the last inequality. As all three sums involved converge, we obtain the following 10-dimensional evaluations:
\begin{align} 5 & \nonumber 
\sum_{m_i} \frac{(-1)^{m_7+m_8+m_9+m_{10}}}{\big[(m_1-\frac12)^2+(m_2-\frac12)^2+(m_3-\frac12)^2+(m_4-\frac12)^2+m_5^2+m_6^2+\cdots+m_{10}^2\big]^4},
\\ \nonumber
 =& \  \sum_{m_i} \frac{(-1)^{m_7+m_8+m_9+m_{10}}}{\big[(m_1-\frac12)^2+(m_2-\frac12)^2+m_3^2+m_4^2+\cdots+m_{10}^2\big]^4}\\
=& \ \frac{1}{8}\, \sum_{m_i} \frac{(-1)^{m_3+m_4+\cdots +m_{10}}}{\big[(m_1-\frac12)^2+m_2^2+m_3^2+\cdots +m_{10}^2\big]^4} = \frac{\Gamma^8(\frac14)}{48\pi^2}. \label{10d1}
\end{align}

Another way to produce 10-dimensional sum evaluations is via $\theta_4^{10}$. It was essentially known to Liouville \cite{liouville} that
\[ \theta_4(q)^{10}-1 = \frac45\bigg[16\sum_{k>0} \frac{k^4 (-q)^k}{1+q^{2k}} + \sum_{k>0} \frac{(-1)^k (2k-1)^4 q^{2k-1}}{1+q^{2k-1}} + 2 \sum_{m,n}(m-ni)^4 (-q)^{m^2+n^2}\bigg]. \]
Taking the Mellin transform \eqref{mellin} to both sides, we find that
\begin{equation} \label{theta10}
{\sum_{m_i}}' \frac{(-1)^{m_1+\cdots+m_{10}}}{(m_1^2+\cdots+m_{10}^2)^s} = -\frac{4}{5}\biggl[\beta(s-4)\eta(s)+16\beta(s)\eta(s-4)-2 {\sum_{m,n}}' \frac{(-1)^{m+n}(m-ni)^4}{(m^2+n^2)^s}\biggr],
\end{equation}
where the notation ${\sum}'$ means the sum is taken over all values of the indices from $-\infty$ to $\infty$, omitting the single term when all the indices are simultaneously 0.

When $s=3$ or $4$, the double sum in \eqref{theta10} can be evaluated in terms of the Eisenstein series $E_4$, see \cite[proof of Theorem~3 and Example~4]{rwz} and also Section \ref{sec:eisen}. Consequently, we have the evaluations
\begin{align} \nonumber
{\sum_{m_i}}' \frac{(-1)^{m_1+\cdots+m_{10}}}{(m_1^2+\cdots+m_{10}^2)^3} & = -\frac{\pi^3}{10}-\frac{\Gamma^8(\frac14)}{120\pi^3}, \\ \label{10d3}
{\sum_{m_i}}' \frac{(-1)^{m_1+\cdots+m_{10}}}{(m_1^2+\cdots+m_{10}^2)^4} & = -\frac{7\pi^4}{1800}-\frac{32 \beta(4)}{5}-\frac{\Gamma^8(\frac14)}{400\pi^2}.
\end{align}

\section{Eisenstein series} \label{sec:eisen}

The Eisenstein series of weight $2n$, $E_{2n}(q)$, is defined by the sum
\begin{equation} \label{eisendef} E_{2n}(q) = \frac{1}{2\zeta(2n)} {\sum_{m_1, m_2}}' \frac{1}{(m_1+ m_2 \tau)^{2n}}, \quad \text{where} \quad q=e^{2\pi i \tau}. \end{equation}
(Here, the notation ${\sum}'$ again means the indices $m_1=m_2=0$ is omitted.)
$E_{2n}$ also admits the Lambert series
\begin{equation} \label{eisen}
 E_{2n}(q) = 1 - \frac{4n}{B_{2n}}\sum_{m=1}^\infty \frac{m^{2n-1}q^m}{1-q^m},
\end{equation}
where $q$ relates to $k$ via \eqref{qderiv} as usual, and $B_i$ denotes the $i$th Bernoulli number.

We state some standards facts about $E_{2n}$. Firstly, it is well-known that 
\begin{equation} \label{e4e6}
 E_{4}(q^2) =\frac{16}{\pi^4}(1-k^2+k^4)K^4, \quad E_6(q^2) = \frac{32}{\pi^6}(1+k^2)(1-2k^2)(2-k^2)K^6.
\end{equation}
(One way to prove this is by noting that both sides are  modular forms of the same weight, and their $q$-expansions agree to sufficiently many terms.)  Moreover, for any integer $n > 1$, $E_{2n}$ can be written as a rational, homogeneous polynomial $P_n(E_4,E_6)$, where each term of the polynomial has weight $2n$. These connections between Eisenstein series and $K$ are exploited in the following theorem.

\begin{theorem} \label{thmkeven}
For any integer $n>1$, there exists a closed form evaluation of the type 
\begin{equation}  \label{thm1eq}
\int_0^1 k\, p_n(k) K^{2n-1-s} K'^{s-1}\mathrm dk = \pi^{2n-1-s}\,\Gamma(s)\zeta(s+1-2n)\zeta(s),
\end{equation}
where $p_n$ is a computable, rational and even polynomial of degree no more than $2n-4$, satisfying $p_n(k)=p_n(k')$.
\end{theorem}

\begin{proof}
We give an effective way to find $p_n$. Our first aim is to cancel out the leading `1' in the Lambert series \eqref{eisen}.

Since $E_{2n}$ can be written in terms of $E_4$ and $E_6$, by \eqref{e4e6}, there exists a rational (computable) polynomial $P_n$ such that 
\begin{equation}\label{en1} E_{2n}(q^2) = P_n\Bigl(\frac{16}{\pi^4}(1-k^2+k^4)K^4, \, \frac{32}{\pi^6}(1+k^2)(1-2k^2)(2-k^2)K^6 \Bigr). \end{equation}
Using \eqref{eq1.3}, we can view $k$ as a function of $q$, namely $k(q) = \theta_2^2(q)/\theta_3^2(q)$. It follows by standard $\theta$ function identities that 
\begin{equation} \label{k12}
 k(q^{1/2}) = \frac{2 \sqrt{k(q)}}{1+k(q)}.
\end{equation}
(This is also the degree 2 modular equation, c.\,f.~\eqref{k1}.) Substituting $q \mapsto q^{1/2}$ in \eqref{en1}, applying \eqref{k12} and simplifying the results using \eqref{k1}, we obtain
\begin{equation}\label{en2}  E_{2n}(q) = P_n\Bigl(\frac{16}{\pi^4}(1+14k^2+k^4)K^4, \, \frac{64}{\pi^6}(1+k^2)(1-34k^2+k^4)K^6 \Bigr). \end{equation}
Similarly, $k(-q) = i k(q)/k'(q)$, and combined with Euler's hypergeometric transformation \cite[Theorem 2.2.5]{aar}, we have
\[ K(k) = \frac{1}{k'} K\left(\frac{ik}{k'}\right). \]
Therefore, equation \eqref{en2} leads to
\begin{equation}\label{en3}  E_{2n}(-q) = P_n\Bigl(\frac{16}{\pi^4}(1-16k^2+16k^4)K^4, \, \frac{64}{\pi^6}(1-2k^2)(1+32k^2-32k^4)K^6 \Bigr). \end{equation}
Now subtract \eqref{en3} from \eqref{en1} and appeal to \eqref{eisen}, the `1'  cancel out and we get
\begin{equation} \label{en4} E_{2n}(q^2)-E_{2n}(-q) = \frac{4n}{B_{2n}} \sum_{j=1}^\infty \frac{j^{2n-1}(-q)^j}{1-q^{2j}} = \pi^{-2n} Q_n(k) K^{2n}, \end{equation}
where $Q_n$ is a degree $2n$ rational polynomial completely determined by $P_n$.

Setting $q = e^{-t}$ in the second term of \eqref{en4}, we apply the Mellin transform \eqref{mellin} to it. Interchanging the order of integration and summation gives
\[ -\frac{4n}{B_{2n}}(1-2^{-s})(1-2^{2n-s})\zeta(s+1-2n)\zeta(s). \]
The corresponding transform of the rightmost term in \eqref{en4} can be found with help from \eqref{qderiv}. Equation \eqref{thm1eq} now follows after some simplifications; note that $p_n$ is determined by $Q_n$, and in turn by $P_n$.

The other claims in the theorem can be proven as follows.  We first show that $p_n$ is a polynomial, which amounts to checking that after multiplying by $\d q/\d k$, the rightmost term in \eqref{en4} remains a polynomial. Thus, we want to prove the claim that $Q_n(k)/(k k'^2)$ is a polynomial. Referring to the definition of $Q_n$ as a difference of two polynomials, we see that the claim is true since at $k=0$ and $\pm 1$, $(1-16k^2+16k^4)$ agrees with $(1-k^2+k^4)$ and $2(1-2k^2)(1+32k^2-32k^4)$ agrees with $(1+k^2)(1-2k^2)(2-k^2)$. The result can be written as $k \, p_n(k)$ since $Q_n$ is even in $k$, so another copy of $k$ can be factored out. The equality $p_n(k)=p_n(k')$ can be verified as the polynomials involved in \eqref{en1} and \eqref{en3} are all invariant under $k \mapsto k'$. Finally, the bound for the degree of $p_n$ holds because $P_n$ preserves weights.
\end{proof}

Note that there are many other ways to cancel out the `1' in \eqref{eisen}; Theorem \ref{thmkeven} simply provides clean results which directly give integrals for \textit{even} powers of $K$.

Using the functional equation for the $\zeta$ function and letting $s \to 1$ in Theorem \ref{thmkeven}, we find that
\begin{equation} \int_0^1 k\,p_n(k) K^{2n-2}\mathrm dk = r_n \, \zeta(2n-1) \end{equation}
for some rational constant $r_n$, while for $s=3, \, 5, \ldots, 2n-3$, the integral in Theorem \ref{thmkeven} vanishes.

\begin{example}
With $n=2$ in Theorem \ref{thmkeven}, we have $E_4 = P_2(E_4, E_6)$ so (obvious) $P_2(x,y)=x$, and $Q_2(k) = 240k^2 k'^2$. Therefore, \eqref{en4} takes the form
\begin{equation} \label{e4case}
  \frac{8}{B_{4}} \sum_{j=1}^\infty \frac{j^{3}(-q)^j}{1-q^{2j}} = 240k^2 k'^2 \, \pi^{-4}K^4.
\end{equation}
After applying the Mellin transform to both sides, we recover \eqref{t2t4}.

Using $E_6$ in Theorem \ref{thmkeven} and $P_3(x,y)=y$, we get
\begin{equation} \label{K4}
\int_0^1 k(1-2k^2)K^{5-s}K'^{s-1}\mathrm{d}k = \frac{\pi^{5-s}}2 \Gamma(s)\eta(s-5)\lambda(s),
\end{equation}
where we have rewritten $\zeta$  in terms of $\eta$ and $\lambda$ according to \eqref{eq1.4}. In the notation of \eqref{eq1.11}, this is an evaluation of $L(4,4,4;s)-2L(8,4,0;s)$.

When $s=-2$, $-1$, 0, 1 or 2 in \eqref{K4}, we obtain
\begin{align} \nonumber
\int_0^1 k(2k^2-1)\frac{K^7}{K'^3} \mathrm{d}k  = \frac{51}{256}\zeta(3)\pi^5, & \quad \int_0^1 k(2k^2-1)\frac{K^6}{K'^2} \mathrm{d}k  = \frac{1905}{64}\zeta(7), \\ \nonumber
\int_0^1 k(2k^2-1)\frac{K^5}{K'} \mathrm{d}k  & = \frac{\log 2}{16}\pi^5, \\
\int_0^1 k(2k^2-1)K^4 \mathrm{d}k  = \frac{93}{16}\zeta(5), & \quad
\int_0^1 k(2k^2-1)K^3K' \mathrm{d}k  = \frac{\pi^5}{128}. \label{k4nice}
\end{align}
The last two integrals in \eqref{k4nice} are particularly interesting;  they were also found by Zhou via very different methods\,--\,see \cite[equation~(63) and last equation]{zhou}. \qede
\end{example}

\begin{example}
Because $E_8 = E_4^2$, we have $P_4(x,y)=x^2$ in the notation of the  proof above. Theorem \ref{thmkeven} then gives
\begin{equation}
\int_0^1 k(2-17k^2+17k^4)K^{7-s}K'^{s-1}\mathrm{d}k = \frac{\pi^{7-s}}4 \Gamma(s)\eta(s-7)\lambda(s).
\end{equation}
In particular, when $s=1$,
\[ \int_0^1 k(2-17k^2+17k^4)K^6\mathrm{d}k = \frac{5715}{64}\zeta(7).\]

Since $E_{10}=E_4E_6$ (so $P_5(x,y)=xy$), we get from Theorem \ref{thmkeven}
\begin{equation}
\int_0^1 k(1-2k^2)(1-31k^2+31k^4)K^{9-s}K'^{s-1}\mathrm{d}k = \frac{\pi^{9-s}}{32} \Gamma(s)\eta(s-9)\lambda(s).
\end{equation}

As another example, the identity $691E_{12} = 441E_4^3+250E_6^2$ gives $P_6$, and leads to 
\begin{align} \nonumber
& \int_0^1 k(2-259k^2+1641k^4-2764k^6+1382k^8)K^{11-s}K'^{s-1}\mathrm{d}k \\
= & \ \frac{\pi^{11-s}}{64}\Gamma(s)\eta(s-11)\lambda(s).
\end{align}
\qede
\end{example}

\section{Modular forms} \label{sec:mod}

Equation \eqref{first} was the first closed form evaluation of the integral of a cubic in $K$ \cite{rwz, wan, zhou} (see also \cite{zhou2}). The proof given in \cite{rwz} relies on the fact that the function
\[ f_4(q) = \sum_{m,n} (m-in)^4 q^{m^2+n^2} \]
is a modular form. We now extend this idea to find closed form evaluations for integrals of \textit{odd} powers of $K$; this complements Theorem \ref{thmkeven} which deals with even powers.

\begin{theorem} \label{thmkodd}
For each integer $p \ge 1$, there exists a closed form evaluation of the type
\begin{equation}
\int_0^1 g_p(k) \, K'^{4p-1} \d k = \frac{\Gamma(\frac14)^{8p}}{\pi^{2p}},
\end{equation}
where $g_p$ is an effectively computable algebraic function.
\end{theorem}

\begin{proof}
Consider the function 
\begin{equation} f_{4p}(q) := \sum_{m,n} (m-in)^{4p} q^{m^2+n^2}. \end{equation}
Note that $f_{4p}$ is real, as the imaginary part vanishes upon summation by symmetry.

Because $(m-in)^{4p}$ is harmonic -- its Laplacian with respect to $m$ and $n$ is 0, it follows from \cite[Part 1C]{zagier} that $f_{4p}$ is a modular (in fact, cusp) form of weight $4p+1$, with non-trivial character, on the congruence subgroup $\Gamma_0(4)$. 

Since $\frac2\pi K = \theta_3^2$ is a weight 1 modular form, it follows that $f_{4p}/K^{4p+1}$ has weight 0, which ensures that it is an algebraic function of $16k^{-2}$, the Hauptmodul for $\Gamma_0(4)$. To summarize, for an algebraic function $G_p$,
\begin{equation} \label{fmodular} f_{4p}(q) = G_p(k)\,  K^{4p+1}. \end{equation}

It remains to compute $G_p$. To do so we expand the real part of the summand of $f_{4p}$ binomially, 
\begin{equation}\label{fbinom} f_{4p}(q) = \sum_{b=0}^{2p} \binom{4p}{2b}(-1)^b \bigg\{\sum_m m^{4p-2b}q^{m^2}\bigg\}\bigg\{ \sum_n n^{2b}q^{n^2}\bigg\}. \end{equation}
Thus, it suffices for us to find an expression for $\sum_n n^{2b}q^{n^2}$, since all terms in the braces are of this form. But $\sum_n n^{2b}q^{n^2}$ can be calculated by applying the operator $q \frac{\d}{\d q}$ to $\theta_3(q)$ a total of $b$ times. To write this in terms of elliptic integrals, we use $\theta_3 = \sqrt{\frac2\pi \, K}$, $q \frac{\d}{\d q} = q  \frac{\d k}{\d q}\frac{\d}{\d k}$, and formulae \eqref{qderiv} and \eqref{derivke}.

This (tedious) computation for $\sum_n n^{2b}q^{n^2}$ produces algebraic functions of $k$, $K$ and $E$, where $E$ appears in the derivatives of $K$ and $E$ by \eqref{derivke}. However, when combined using \eqref{fbinom} to yield the final expression for $f_{4p}$, all the $E$'s are guaranteed to cancel out due to \eqref{fmodular}, and this combination gives us $G_p$.

We now take the Mellin transform of both sides of \eqref{fmodular}, with $q=e^{-t}$ and $s=4p$. The left hand side simplifies to an Eisenstein series $E_{4p}$ with $\tau = i$, which is a (computable) rational constant times a power of 
\[ E_4 \Big| _{k = 1/\sqrt2} = \frac{3\Gamma^8(\frac14)}{64\pi^6}, \]
as $E_6|_{k=1/\sqrt2}=0$ (see \eqref{e4e6} and surrounding discussion for these connections). The theorem now follows after some algebraic manipulations.
\end{proof}

\begin{example}
In Theorem \ref{thmkodd}, taking $p=1$ leads to \eqref{first}. For $p=2$, following the steps in the proof of the theorem, we have
\[ f_8(q) =  \frac{32k^2 k'^2 (4+k^2-k^4)}{\pi^9}K^9, \]
from which we obtain
\begin{equation} 
\int_0^1 k(4+k^2-k^4)K'^7\mathrm{d}k = \frac{3\, \Gamma^{16}(\frac14)}{2^{12}\,5\,\pi^4}.
\end{equation}
Similarly, for $p=3$, 
\begin{equation} 
\int_0^1 k(16-92k^2+93k^4-2k^6+k^8)K'^{11}\mathrm{d}k = \frac{189\, \Gamma^{24}(\frac14)}{2^{15}\,65\,\pi^6},
\end{equation}
and for $p=4$,
\begin{equation} 
\int_0^1 k(64+848k^2-2136k^4+2577k^6-1291k^8+3k^{10}-k^{12})K'^{15}\mathrm{d}k = \frac{43659\, \Gamma^{32}(\frac14)}{2^{21}\,85\,\pi^8}.
\end{equation}
It seems that in all cases, $g_p$ is actually an odd polynomial.
\qede
\end{example}

\begin{remark} \label{remarkodd}
We can exploit equation \eqref{e4e6} in many more ways to produce $K$ integrals. As one example, starting with
\[ g(q) = \sum_{m,n} \big(n-\sqrt{2}im\big)^4 q^{n^2+2m^2} = \sum_{m,n} (n^4-12m^2n^2+4m^4) q^{n^2+2m^2}, \]
we use the procedure (repeated applications of $q \frac{\d}{\d q}$) outlined in the proof of Theorem \ref{thmkodd} to write this as
\[ g(q) = \frac{\sqrt2}{\pi^5} \frac{k^2 k' [k^2(k'-2)+4(k'+1)]K^5}{(k'+1)^{3/2}}, \] 
where we have used $\sum_m q^{2m^2} = \theta_3(q^2) = \sqrt{\frac{1+k'}{\pi}K}$.
We now let $q = e^{-t}$ and apply the Mellin transform \eqref{mellin} with $s=4$ to both expressions. The resulting double sum,
\[ {\sum_{m,n}}' \frac{1}{(n+\sqrt 2\, i m)^4}, \]
can be evaluated using \eqref{e4e6} and a \textit{singular value} of $K$ (see \cite[Ch.~4]{AGM}), namely the closed forms
\begin{equation} \label{singular2}
 k(e^{-\sqrt 2 \pi}) = \sqrt2-1, \qquad K\big(k(e^{-\sqrt 2 \pi})\big) =  \sqrt{\frac{2+\sqrt2}{128\pi}}\,\Gamma\Big(\frac18\Big)\Gamma\Big(\frac38\Big). 
\end{equation}
After simplification, we produce the integral evaluation
\[ \int_0^1 \frac{2+3k-k^2}{\sqrt{1+k}}K^3\mathrm dk = \frac{\Gamma^4(\frac18)\Gamma^4(\frac38)}{384\sqrt2 \pi^2}. \]
Exactly the same procedure may be used to relate 
\[ {\sum_{m,n}}' \frac{1}{(n+\sqrt 2 i m)^{2p}}, \quad p = 2,3,4,\ldots\]
which has a closed form (by \eqref{e4e6}, subsequent discussion, and \eqref{singular2}), to an integral involving $K^{2p-1}$. For instance, when $p=3$, we have
\begin{equation} \label{first5}
\int_0^1 \frac{4-6k+5k^2+12k^3+k^4}{\sqrt{1+k}}K^5\mathrm dk = \frac{\Gamma^6(\frac18)\Gamma^6(\frac38)}{2304\sqrt2 \pi^3}.
\end{equation}
Indeed, we may obtain closed form evaluations for integrals containing $K^{4p+1}$ \textit{for all} $p \in \mathbb{N}$ this way, with \eqref{first5} being the first known example; this is a counterpart to Theorem \ref{thmkodd}, which deals with $K^{4p-1}$.

Here, we record a few more integrals obtained using variations of this method:
\begin{align}\nonumber
\int_0^1 \frac{1+14k^2+k^4}{\sqrt k}K'^5\mathrm dk =32\int_0^1 \frac{1+14k^2+k^4}{\sqrt k}K^5\mathrm dk & = \frac{3\Gamma^{12}(\frac14)}{32\pi^3}, \\ \nonumber
\int_0^1 (1-k^2-4k^4)K'^7\mathrm dk = -\frac{120}7\int_0^1 (1-k^2-4k^4)K^7\mathrm dk & = \frac{9\Gamma^{16}(\frac14)}{4096\pi^4},\\
\int_0^1 \frac{(1+14k^2+k^4)^2}{\sqrt k}K'^9\mathrm dk =512\int_0^1 \frac{(1+14k^2+k^4)^2}{\sqrt k}K^9\mathrm dk & = \frac{189\Gamma^{20}(\frac14)}{128\pi^5}.
\end{align}
\qede
\end{remark}

\begin{remark} \label{rem5}
Though the details are omitted here, we should mention that it is fruitful to consider more general Eisenstein series than the definition \eqref{eisendef}; for instance, one generalization involves twisting the numerator 1 in \eqref{eisendef} by $\chi_a(m_1)\chi_b(m_2)$, where $\chi_a, \, \chi_b$ are Dirichlet characters. At suitable $\tau$, modular theory tells us that the ratio of such a construction over $E_{2n}$ is algebraic (and computable). As one example, using $\chi_{-4}(m_1)\chi_{-8}(m_2)$ (whose values are implicitly given by $L_{-4}$ and $L_{-8}$ in \eqref{eq1.4}), we can use the procedure in the proof of  Theorem \ref{thmkodd} to ultimately deduce
\begin{equation} \int_0^1 k^{1/4}k'^{1/2}K^3\mathrm dk = \frac{(\sqrt2-1)^{3/2}}{128\sqrt2\,\pi^2}\Gamma^8\Big(\frac14\Big), \end{equation}
which, by \eqref{eq1.13} and \cite[equation (1.3.8)]{Lbook}, can be converted into the 10-dimensional sum
\[ \sum_{m_i, n_i} \frac{(-1)^{m_1+\cdots+m_5}}{\big[2m_1^2+\cdots +2m_5^2 + (n_1-\frac12)^2+\cdots +(n_5-\frac12)^2\big]^4} = \frac{(\sqrt2-1)^{3/2}}{48\sqrt2\,\pi^2}\Gamma^8\Big(\frac14\Big). \]
\qede
\end{remark}

\section{Examples involving $E$}

Integrals involving $E$ can be obtained by differentiating $q$-identities involving $K$. We give one worked example here.

As usual, with $k = \theta_2^2(q)/\theta_3^2(q)$, we differentiate both sides of \eqref{e4case} with respect to $q$. With the help of equations \eqref{derivke} and \eqref{qderiv}, we obtain
\[ \sum_{j=1}^\infty \frac{j^4(1+q^{2j}) (-q)^j}{(1-q^{2j})^2} = -\frac{4k^2 k'^2}{\pi^6}K(k)^5 (2E(k)-K(k)).\]
Next, with $q=e^{-t}$, we compute the Mellin transform \eqref{mellin} of both sides above. The result is
\begin{equation} \label{eeg}
 \int_0^1 k K'^s K^{3-s}(2E-K)\mathrm dk = \frac{\pi^{4-s}}{2}\Gamma(s+1)\eta(s-3)\lambda(s), 
\end{equation}
which, after applying \textit{Legendre's relation} $EK'+E'K-KK'=\frac{\pi}2$, generalizes an identity in \cite[Remark after proof of Prop.~5.1]{zhou}. 

It is clear that this procedure can be applied to every $K$ integral in this paper. We remark that there are many other ways to construct $E$ integrals, relying on for instance Legendre's relation, though we do not pursue such paths here; instead, only some aesthetically pleasing examples are mentioned below.

\begin{example}
It is known (see e.\,g.~\cite{zucker}) that
\[ \frac{4}{\pi^2} k^2 K^2 =\theta_2^4(q) = 16 \sum_{n=0}^\infty \frac{(2n+1)q^{2n+1}}{1-q^{4n+2}}, \]
so we can differentiate both sides with respect to $q$ and take the Mellin transform, ending up with  
\[\int_0^1 \frac{E'}{k}\frac{K^s}{K'^{s-1}}\mathrm dk = 2\pi^{2-s} \Gamma(s+1)\lambda(s-1)\lambda(s). \]
Similarly, starting from the $q$-series for $\theta_2^2\theta_4^4$ \cite{zucker}, we get
\[ \int_0^1 K^{2-s} K'^s(3E-2K)\mathrm dk = \frac{\pi^{3-s}}2 (2^s-1)\Gamma(s+1)\beta(s-2)\zeta(s), \]
which, when combined with \eqref{first}, has the special evaluation at $s=3$:
\[  \int_0^1 \frac{E K'^3}{K}\mathrm dk = \frac{\Gamma^8(\frac14)}{192\pi^2}+\frac{7\pi}{4}\zeta(3). \]
Starting from the $q$-series for $\theta_2^3\theta_4^3$ \cite{zucker}, we can obtain
\[ \int_0^1 \sqrt{\frac{k}{k'}}K'^2 (2E-K)\mathrm d k= \frac{\pi^3}{6\sqrt2}.\]
Our final example involves double differentiation. We apply the operator $q \frac{\mathrm d }{\mathrm dq}$ twice to the identity
\[ \frac{2}{\pi} k K =  \theta_2^2(q) = 4q^{1/2} \sum_{n=0}^\infty \frac{q^n}{1+q^{2n+1}}, \]
followed by taking the Mellin transform. After simplification, we have
\begin{equation}
 \int_0^1 \frac{2E'^2-k^2 K'^2}{kk'}K^s K'^{1-s}\mathrm dk = 2^{s-1}\pi^{3-s}\Gamma(s+1)\lambda(s-1)\beta(s-1).
\end{equation}
At $s=3$, aided by \eqref{first} we deduce
\[ \int_0^1 \frac{E'^2 K^3}{kk'K'^2}\mathrm dk = \frac{3}{2}\pi^2G+\frac{\Gamma^8(\frac14)}{256\pi^2}. \]
\qede
\end{example}

\section{Conclusion}

We have only scratched the surface of the many rich identities that can be found. Even though our methods do not simplify all integrals of the form \eqref{eq1.2}, we have reduced many sets of them into known multiple sum evaluations, and hence results such as \eqref{first} no longer seem mysterious or isolated. Conversely, we have produce new  lattice sums evaluations from $K$ integrals, for instance in Example \ref{eg5} and Section \ref{sec:10d}.

Among the new $K$ integrals produced, we have shown that for each $n \in \mathbb{N}$, there is at least one computable algebraic function $r_n$ such that 
\[ \int_0^1 r_n(k) K(k)^n\mathrm dk \]
can be evaluated in closed form; the proof of this general result is achieve through Theorem \ref{thmkeven}, Theorem \ref{thmkodd}, and Remark \ref{remarkodd}.

\acknowledgements We would like to thank Jon Borwein, Heng Huat Chan and Armin Straub for useful discussions.

\newpage

\begin{center}
{ \renewcommand{\arraystretch}{1.75}
\begin{longtable}{|c|c|c|}   \hline 
\textbf{Source} & \textbf{Lattice sum} & \textbf{Integral} \\ \hline
$\theta_2$ & $2^{2s+1}\lambda(2s)$  & $ \f{\pi^{s+1/2}}{\Gamma(s)}\i1 \f{1}{\sqrt{2k}k'^2K^{3/2}}\k$ \\ \hline
$(\theta_2\theta_4)^{1/2}$ & $2^{3s+1/2}L_8(2s)$  & $ \f{\pi^{s+1/2}}{\sqrt{2}\Gamma(s)}\i1 \f{1}{k^{3/4}k'^{7/4}K^{3/2}}\k$ \\ \hline
$(\theta_2\theta_3\theta_4)^{1/3}$ & $2^{1/3}12^s L_{12}(2s)$  & $ \f{\pi^{s+1/2}}{\sqrt{2}\Gamma(s)}\i1 \f{1}{k^{5/6}k'^{11/6}K^{3/2}}\k$ \\ \hline
$\theta_3^{1/2}(\theta_2\theta_3\theta_4)^{1/6}$ & $2^{1/6}24^s L_{24}(2s)$  & $  \f{\pi^{s+1/2}}{\sqrt2\Gamma(s)}\i1 \f{1}{k^{11/12}k'^{23/12}K^{3/2}}\k$ \\ \hline
$\theta_2^2$ & $2^{s+2}\lambda(s)\beta(s)$  & $ \f{\pi^s}{\Gamma(s)}\i1 \f{1}{k'^2K}\k$ \\ \hline
$\theta_2\theta_3$ & $2^{2s+1}\lambda(s)\beta(s)$ & $ \f{\pi^s}{\Gamma(s)}\i1 \f{1}{\s{k}k'^2K}\k$ \\ \hline
$\theta_2\theta_4$ &  $2^{2s+1}L_8(s)L_{-8}(s)$ & $\f{\pi^{s}}{\Gamma(s)}\i1 \f{\s{kk'}}{kk'^2K}\k$ \\  \hline
$\theta_2(q^2)\theta_3$ &  $2^{s+1}\lambda(s) L_{-8}(s)$ & $\f{\pi^s}{\Gamma(s)} \int_0^1 \frac{\sqrt{(1-k')/2}}{kk'^2K} \k$ \\ \hline
$\theta_2(q^2)\theta_4$ &  $2^{s+1}\beta(s) L_{8}(s)$ & $\f{\pi^s}{\Gamma(s)} \int_0^1 \frac{\sqrt{(1-k')/2}}{kk'^{3/2}K} \k$ \\ \hline
$\theta_3^2-\theta_3\theta_3(q^2)$ &  $2\zeta(s)(2\beta(s)-L_{-8}(s))$ & $\f{\pi^s}{\Gamma(s)} \int_0^1 \frac{1-\sqrt{(1+k')/2}}{kk'^2K} \k$ \\ \hline
$\theta_3^2-\theta_4^2$ &  $8\lambda(s)\beta(s)$ & $\f{\pi^s}{\Gamma(s)}\i1 \f{1-k'}{kk'^2K}\k$ \\ \hline
$\theta_2\theta_3\theta_4$ & $2^{2s+1}\beta(2s-1)$ & $\f{\sqrt{2}\pi^{s-1/2}}{\Gamma(s)}\i1 \f{1}{\s{k k'^3 K}}\k$ \\ \hline
$\theta_2\theta_3(\theta_2\theta_3\theta_4)^{1/3}$ & $3^{s}(1+2^{2-2s})L_{-3}(2s-1)$  &  $ \f{\pi^{s-1/2}}{2^{5/6}\Gamma(s)}\i1 \f{1}{k^{1/3}k'^{11/6}\s{K}}\k$ \\ \hline
$\theta_2\theta_4(\theta_2\theta_3\theta_4)^{1/3}$ & $3^{s}L_{-3}(2s-1)$ & $ \f{\pi^{s-1/2}}{2^{5/6}\Gamma(s)}\i1 \f{1}{k^{1/3}k'^{4/3}\s{K}}\k$ \\ \hline
$\theta_3^{2}(\theta_2\theta_4)^{1/2}$ &  $8^{s}L_{-8}(2s-1)$ & $ \f{\pi^{s-1/2}}{\Gamma(s)}\i1 \f{1}{k^{3/4}k'^{7/4}\s{K}}\k$ \\ \hline
$\theta_3^{5/2}(\theta_2\theta_3\theta_4)^{1/6}$ & $24^{s}L_{-24}(2s-1)$  &  $ \f{\pi^{s-1/2}}{\Gamma(s)}\i1 \f{2^{1/3}}{k^{11/12}k'^{23/12}\s{K}}\k$ \\ \hline
$\theta_4^{5/2}(\theta_2\theta_3\theta_4)^{1/6}$ & $24^{s}(1+2^{1-2s})L_{-3}(2s-1)$  &  $  \f{\pi^{s-1/2}}{\Gamma(s)}\i1 \f{2^{1/3}}{k^{11/12}k'^{2/3}\s{K}}\k$ \\ \hline
$\theta_2^4$ & $16\lambda(s) \lambda(s-1)$ & $\f{2\pi^{s-1}}{\Gamma(s)}\i1 \f{k}{k'^2}\k$ \\ \hline
$\theta_2^3\theta_3$ &  $4^s\left[\lambda(s)\lambda(s-1)-\beta(s)\beta(s-1)\right]$ &  $\f{2\pi^{s-1}}{\Gamma(s)}\i1 \f{\s{k}}{k'^2}\k$ \\ \hline
$\theta_2^2\theta_3^2$ &  $2^{s+2}\lambda(s) \lambda(s-1)$ &  $\f{2\pi^{s-1}}{\Gamma(s)}\i1 \f{1}{k'^2}\k$ \\ \hline
$\theta_2\theta_3^3$ &  $4^s\left[\lambda(s)\lambda(s-1)+\beta(s)\beta(s-1)\right]$ & $\f{2\pi^{s-1}}{\Gamma(s)}\i1 \f{\s{k}}{kk'^2}\k$ \\ \hline
$\theta_2^3\theta_4$ &  $4^s\left[L_{-8}(s)L_{-8}(s-1)-L_{8}(s)L_{8}(s-1)\right]$ & $\f{2\pi^{s-1}}{\Gamma(s)}\i1\s{ \f{k}{k'^3}}\k$ \\ \hline
$\theta_2^2\theta_4^2$ &  $2^{s+2}\beta(s)\beta(s-1)$ &  $\f{2\pi^{s-1}}{\Gamma(s)}\i1 \f{1}{k'}\k$ \\ \hline
$\theta_2\theta_4^3$ &  $4^s\left[L_{-8}(s)L_{-8}(s-1)+L_{8}(s)L_{8}(s-1)\right]$ & $\f{2\pi^{s-1}}{\Gamma(s)}\i1\f{1}{\sqrt{kk'}}\k$ \\ \hline
$\theta_3^4 - \theta_3^2\theta_4^2$  & $8(1+2^{2-s})\lambda(s)\lambda(s-1)$ & $\frac{2\pi^{s-1}}{\Gamma(s)} \int_0^1 \frac{k}{k'^2(1+k')} \k$  \\ \hline
$\theta_2^6$ &  $2^{s+2}\left[\lambda(s-2)\beta(s)-\beta(s-2)\lambda(s)\right]$ &  $\f{4\pi^{s-2}}{\Gamma(s)}\i1 \f{k^2K}{k'^2}\k$ \\ \hline
$\theta_2^4\theta_3^2$ &  $16\zeta(s-2)\beta(s)$ &  $\f{4\pi^{s-2}}{\Gamma(s)}\i1 \f{kK}{k'^2}\k$ \\ \hline
$\theta_2^3\theta_3^3$ &  $2^{2s-1}\bigl[\lambda(s-2)\beta(s)-\beta(s-2)\lambda(s)\bigr]$ &  $\f{4\pi^{s-2}}{\Gamma(s)}\i1 \f{\s{k}K}{k'^2}\k$ \\ \hline
$\theta_2^2\theta_3^4$ &  $2^{s+2}\beta(s)\lambda(s-2)$ &  $\f{4\pi^{s-2}}{\Gamma(s)}\i1 \f{K}{k'^2}\k$ \\ \hline
$\theta_2^4\theta_4^2$ &  $16\eta(s-2)\beta(s)$ &    $\f{4\pi^{s-2}}{\Gamma(s)}\i1 \f{kK}{k'}\k$ \\ \hline
$\theta_2^3\theta_4^3$ & $  2^{2s-1}\bigl[L_{8}(s)L_{-8}(s-2)-L_{8}(s-2)L_{-8}(s)\bigr]$ & $\f{4\pi^{s-2}}{\Gamma(s)}\i1\s{ \f{k}{k'}}K\k$ \\ \hline
$\theta_2^2\theta_4^4$ &  $2^{s+2}\beta(s-2)\lambda(s)$ &  $\f{4\pi^{s-2}}{\Gamma(s)}\i1 K\k$ \\ \hline
$\theta_3^6-\theta_4^6$ &  $32\beta(s)\lambda(s-2)-8\beta(s-2)\lambda(s)$ &  $\f{4\pi^{s-2}}{\Gamma(s)}\i1 \f{(1-k'^3)K}{kk'^2}\k$  \\ \hline
$\theta_4^6-\theta_3^2\theta_4^4$  & $8\lambda(s)\beta(s-2)-16\eta(s-2)\beta(s)$ & $\frac{4\pi^{s-2}}{\Gamma(s)}\int_0^1 \frac{k'-1}{k}K\k$ \\ \hline
$\theta_2^8$ & $2^{8-s}\zeta(s-3)\lambda(s)$ & $\f{8\pi^{s-3}}{\Gamma(s)}\i1\f{k^3}{k'^2}K^2\k$ \\ \hline
$\theta_2^4\theta_3^4$ &  $16\zeta(s-3)\lambda(s)$ &  $\f{8\pi^{s-3}}{\Gamma(s)}\i1\f{k}{k'^2}K^2\k$ \\ \hline
$\theta_2^4\theta_4^4$ &  $16\eta(s-3)\lambda(s)$ & $\f{8\pi^{s-3}}{\Gamma(s)}\i1kK^2\k$  \\ \hline 
$\theta_3^8-\theta_4^8$ &  $32\lambda(s)\lambda(s-3)$ &  $\f{8\pi^{s-3}}{\Gamma(s)}\i1 \f{k(2-k^2)K^2}{k'^2}\k$ \\ \hline
\caption{Integrals of $K$ from various $\theta$ products.} \label{table1} 
\end{longtable} }
\end{center}

\end{document}